\documentclass[12pt]{amsart}
\usepackage{amsmath,amssymb,amscd,amsthm,amsxtra}
\usepackage{latexsym}
\usepackage[colorlinks=true, pdfstartview=FitV, linkcolor=blue, citecolor=blue, urlcolor=blue]{hyperref}
\usepackage{color}

\usepackage{graphicx,epstopdf,verbatim}

\newtheorem{theorem}{Theorem}[section]
\newtheorem{lemma}[theorem]{Lemma}
\newtheorem{proposition}[theorem]{Proposition}
\newtheorem{remark}[theorem]{Remark}

\newtheorem{definition}{Definition}[section]
\newtheorem{corollary}[theorem]{Corollary}

\newcommand{\q}{\quad}
\newcommand{\qq}{\quad\quad}
\newcommand{\qqq}{\quad\quad\quad}
\newcommand{\qqqq}{\quad\quad\quad\quad}

\newcommand{\al}{\alpha}

\newcommand{\de}{\delta}

\newcommand{\ez}{\epsilon}

\newcommand{\la}{\lambda}

\newcommand{\si}{\sigma}

\newcommand{\cm}{\mathcal M}
\newcommand{\cb}{\mathcal B}

\newcommand{\cd}{\mathcal D}
\newcommand{\cq}{\mathcal Q}

\newcommand{\crec}{\mathcal R}

\newcommand{\f}{\frac}

\newcommand{\nf}{\infty}
\newcommand{\rrr}{\mathbf R}

\newcommand{\rn}{\mathbf R^n}

\def\rr{{\mathbb R}}
\def\nn{{\mathbb N}}
\def\lf{\left}
\def\r{\right}

\numberwithin{equation}{section}

\begin{document}

\title[The multilinear strong maximal function]{The multilinear strong maximal function}

\author{Loukas Grafakos, Liguang Liu, Carlos P\'erez, Rodolfo H.  Torres}

\begin{abstract}
A multivariable version of the strong maximal function is introduced and
a sharp distributional estimate for this operator in the spirit of the Jessen, Marcinkiewicz,
and Zygmund theorem is obtained. Conditions that characterize
the boundedness of this multivariable operator on
products of weighted Lebesgue spaces
equipped with multiple weights are obtained.
Results for other multi(sub)linear maximal functions associated with bases of open
sets are studied too.    Bilinear interpolation  results between distributional
estimates, such as those satisfied by the multivariable  strong maximal function,  are also proved. 
\end{abstract}

\address{Loukas Grafakos\\
Department of Mathematics\\
University of Missouri\\
Columbia, MO 65211, USA } \email{grafakosl@missouri.edu}

\address{Liguang Liu\\
Department of Mathematics\\
School of Information\\
Renmin University of China\\
Beijing 100872, China }
\email{liguangbnu@gmail.com}

\address{Carlos P\'erez\\
Departamento De An\'alisis Matem\'atico, Facultad de Matem\'aticas,
Universidad De Sevilla, 41080 Sevilla, Spain.}
\email{carlosperez@us.es}

\address{Rodolfo H. Torres\\
Department of Mathematics, University of Kansas, 405 Snow Hall 1460
Jayhawk Blvd, Lawrence, Kansas 66045-7523, USA.}
\email{torres@math.ku.edu}

\date{\today}



\thanks{ The authors  would like to acknowledge  the support of the following grants.
First author: NSF grant DMS 0900946.  Third author: Spanish Ministry
of Science and Innovation grant MTM2009-08934. Fourth author:  NSF
grant DMS 0800492 and a General Research Fund allocation of the
University of Kansas. }\maketitle



\section{Introduction}
Maximal functions have proved to be  tools of great importance in
harmonic analysis. Their study not only contains intrinsic interest
but also intertwines with the study of singular integral operators,
most notably in the context of weighted norm inequalities.
Maximal functions have also
multi(sub)linear versions 
that play an equally important role in the study of multilinear operators.
Some fundamental linear results are no longer readily available in the
multilinear setting, but this shortfall  puts in evidence new
interesting phenomena that lead to further investigation. Finding
appropriate substitutes and alternative tools in the study of
multilinear problems is an ongoing project undertaken by
a number of 
researchers. The purpose of this article is to contribute to this
endeavor by studying multilinear versions of the strong maximal
function and intimately related topics.

Some of the motivation for our work arises from the recent article
by Lerner et al \cite{LOPTT} where the multisublinear maximal function
\begin{equation}\label{first11}
\cm (f_1,\dots , f_m)(x)=\sup_{\substack{Q\ni x\\
Q\,\,\textup{cube}} } \prod_{i=1}^m\frac{1}{|Q|}\int_Q|f_i(y_i)|\,
dy_i
\end{equation}
associated with cubes with sides parallel to the coordinate axes was
introduced. This maximal function
is an analogue of the Hardy--Littlewood maximal function and 
led to the characterization of the
class of multiple weights for which multilinear Calder\'on-Zygmund
operators are bounded on products of weighted Lebesgue spaces.
Such operators were  introduced by Coifman and Meyer in \cite{CM1}, \cite{CM2}
and were systematically studied in Grafakos and Torres \cite{GT4} (see also the references therein).

The boundedness
$$
\cm: L^{p_1}(\rn) \times  \dots \times L^{p_m}(\rn)  \to L^p(\rn)
$$
whenever
$$
1<p_1, \dots, p_m \le \nf  \qq \mbox{and}  \qq  \f1p=\f1{p_1}+ \dots +\f1{p_m}\, ,
$$
is a simple consequence of  H\"older's inequality and of the trivial observation that
\begin{equation}\label{e4}
\cm(f_1,\dots , f_m)(x) \le \prod_{i=1}^m M(f_i) \, ,
\end{equation}
where $M$ stands for the classical Hardy--Littlewood maximal function.
Also, H\"older's inequality  for weak spaces yields the appropriate endpoint boundedness, namely
$$
\cm: L^{1}(\rn) \times  \dots \times L^{1}(\rn)  \to L^{\f1m,\infty}(\rn)\, .
$$

In this article we study corresponding estimates for
the {\it $m$-sublinear version of the strong maximal function}, or simply (with a certain abuse of terminology)
{\it the strong multilinear maximal function}. We define this operator as
$$
\cm_{\crec} (\vec f\,)(x)= \sup_{R\ni x} \prod_{i=1}^m \lf(\f 1{|R|}\int_R |f_i(y)|\,dy\r)
\q x \in \rn\, ,
$$
where $\vec f =(f_1,\dots,f_m)$ is an  $m$-dimensional vector of
locally integrable functions and where the supremum is taken over
all rectangles with sides parallel to the coordinate axes. Like
$\cm$, the strong multilinear maximal functions is controlled by the
$m$-fold tensor product of the  maximal function  of each variable.
That is,
\begin{equation}\label{e3}
\cm_{\crec}(\vec f\,) \le \prod_{i=1}^m M_{\crec}(f_i)\, ,
\end{equation}
where $M_{\crec}$ denotes the strong maximal operator on $\rn$ given by
\begin{equation}\label{StrongMF}
M_{\crec}(f)(x)=\sup_{R\ni x} \f1{|R|}\int_R |f(y)|\,dy\,  , 
\end{equation}
 where  the supremum is taken over all rectangles with
sides parallel to the coordinate axes. Obviously, \eqref{e3}
together with H\"older's inequality, yields the appropriate strong
type boundedness
$$
\cm_{\crec}: L^{p_1}(\rn) \times  \dots \times L^{p_m}(\rn)  \to
L^p(\rn)
$$
whenever
$$
1<p_1, \dots, p_m \le \nf  \qq \mbox{and}  \qq  \f1p=\f1{p_1}+ \dots
+\f1{p_m}\, .
$$
However, $M_{\crec}$ is not of weak type $(1,1)$; 
as a substitute, Jessen, Marcinkiewicz and Zygmund \cite{jmz} showed that there is a constant $C_n$ depending
only on the dimension $n$ such that  for all $f $ on $\rn$,
\begin{equation*}
|\{x\in\rn:\, M_{\crec}(f)(x)>\la\}| \le C_n\int_{\rn}
\Phi_n\lf(\f{|f(x)|}{\la}\r)\,dx,\,  
\end{equation*}
where for $ t>0$,
$$
\Phi_n(t)=t(1+(\log^+t)^{n-1}) \approx  t(\log(e+t))^{n-1}.
 $$
 Unlike the case of cubes, in which the classical weak type endpoint  of  $M$ immediately extends to an estimate for its $m$-fold product on $L^1\times \dots \times L^1$ and hence to $\cm$,   it is not clear how to derive directly an appropriate endpoint  estimate
 for the $m$-fold product of  $M_{\crec}$ or even $\cm_{\crec}$. Part of the problem stems from the fact    that the Jessen, Marcinkiewicz and Zygmund result is a so-called modular estimate and not a norm estimate.

The geometry of rectangles in $\rn$ is
more intricate than that of cubes in $\rn$,  even when both classes
of sets are restricted to have sides parallel to the axes. The
failure of the engulfing property of intersecting rectangles
presents a crucial difference between these two classes and signals
that  the endpoint behavior of the strong maximal function is
fundamentally different than that of the Hardy-Littlewood maximal
operator. Delicate properties of rectangles in $\rn$
that still make possible the analysis  were quantified by C\'ordoba
and Fefferman \cite{cf} in their alternative geometric proof of the
classical Jessen, Marcinkiewicz, and Zygmund \cite{jmz} endpoint
estimate for the strong maximal function. We investigate the
analogous situation in the multilinear setting and look at endpoint
distributional estimates. An interesting point is that our result in
this direction, Theorem \ref{BSMF}, says that the operator is more
singular as $m$ increases.  In addition to the natural interest of
the strong maximal function, part of our motivation to study such
endpoint results arose also from a question posed in \cite{LOPTT}
about the analysis of multilinear commutators via multilinear
interpolation involving distributional estimates.

Obtaining multilinear estimates from linear ones
produces far from optimal results when a theory of weights  is
considered. Indeed, the pointwise estimate \eqref{e4} is not sharp
enough to be used as the starting point for the development of the
relevant theory for  the maximal function $\cm$. The work in
\cite{LOPTT} put in evidence that there is a much larger class of
weights that characterizes the boundedness of $\cm$, as well  as the
boundedness of  singular integrals and commutators,  than that
previously considered by Grafakos and Torres \cite{GT5} and P\'erez
and Torres \cite{PT}.

The  relevant class of  multiple  weights for $\cm$ is given by the
$A_{\vec P}$ condition: for $\vec P = (p_{1},\cdots, p_{m})$,\,with
$1<p_1, \dots , p_{m}<\infty$, $\vec w \in A_{\vec P}$ if
\begin{equation}\label{multiapCubes}
\sup_{Q}\frac{1}{|Q|}\int_Q\nu_{\vec
w}(x)\,dx\,\prod_{j=1}^m\Big(\frac{1}{|Q|}\int_Q
w_j(x)^{1-p'_j}\,dx\Big)^{\frac{p}{p'_j}}<\infty
\end{equation}
where $\frac{1}{p}=\frac{1}{p_1}+\dots+\frac{1}{p_m}$ and
$$\nu_{\vec w}=\prod_{j=1}^mw_j^{p/p_j}.$$
%
For $m=1$,   \eqref{multiapCubes} goes back to the well-known Muckenhoupt $A_p$ condition. It should be remarked that $$\prod_{j=1}^m A_{p_j}\subset A_{\vec P}$$
with strict  inclusion; see \cite[p.1232 and Remark~7.2]{LOPTT}.
We address similar questions involving $\cm_{\crec}$  and more
general maximal functions $\cm_{\cb}$ associated with other bases
$\cb$.  To do so we need to introduce appropriate new classes of
multiple weights; we do this in Section \ref{generalbasis}.

In the linear case, one may often obtain strong weighted estimates
from weak type ones using  interpolation and  the reverse H\"older's property of weights.
This approach was adapted by P\'erez \cite{P1} for general maximal functions.
However, a certain complication in the multilinear setting arises
and it does not seem to be possible to use interpolation to pass from
weak to strong estimates for the aforementioned classes of weights.
This complication is bypassed in this article by directly proving strong type estimates  using
delicate 
techniques for maximal functions on general  bases of open sets  
 adapted from 
the work of Jawerth and Torchinsky \cite{JT} and  Jawerth \cite{J}.

The main results in this article are the following:

1. A characterization of all $m$-tuples of weights $\vec w
=(w_1,\dots , w_m)$ for which the  multilinear strong maximal
function maps
$$
 L^{p_1}(w_1)   \times \cdots
\times L^{p_m}(w_m) \to L^{p}(\nu_{\vec w})$$
where $\nu_{\vec w}=\prod_{j=1}^mw_j^{p/p_j}$, $1<p_1,\dots , p_m<\nf $,
and $\f1p=\f1{p_1}+\dots + \f1{p_m}$.  This characterization requires the notion of the
multilinear $A_{\vec P}$ condition adapted to rectangles with sides parallel to the
axes in $\rn$ and is contained in Theorem \ref{1wghtStrChar}.

2.  The \,$ L^{p_1}(w_1)
 \times \cdots \times L^{p_m}(w_m) \to L^{p,
\nf}(\nu)$\, boundedness of the multilinear strong maximal function,
whenever the weights $w_1, \dots , w_m$ and an arbitrary $\nu$
satisfy a certain power bump   variant of the
 multilinear $A_p$ condition.
This is given in Theorem~\ref{twoweights}, from which
a characterization of the weak type inequality in the
case $\nu = \nu_{\vec w}$ also follows, see Corollary~\ref{weakoneweight}.

3.  A sharp distributional estimate for the multilinear strong
maximal operator, analogous to that of Jessen, Marcinkiewicz, and
Zygmund. This can be found in Theorem~\ref{BSMF}. We also note that
the type of multilinear endpoint distributional estimate obtained is
very suitable for purposes of interpolation. In fact, we present in
Theorem~\ref{interpolation2} a more general version of the bilinear
Marcinkiewicz interpolation theorem. We prove strong type bounds for
a certain range of exponents starting from  weak type bounds and
a distributional estimate like the one that  the strong maximal
function satisfies.

 To facilitate the reader's access to each of the independent
results contained herein, this article is organized as follows.  A
discussion of the classical weighted theory and its multilinear
extension with respect to general bases is given in Section
\ref{generalbasis}; this also contains the statements of
Theorem~\ref{twoweights} and Theorem~\ref{1wghtStrChar}.  The proof
of the weak boundedness of the strong maximal function in the case
of $m+1$ weights (Theorem~\ref{twoweights})
  is postponed until Section \ref{two-weight}, while
the proof of the strong boundedness in the case of $m$ weights
(Theorem~\ref{1wghtStrChar}) is given in Section
\ref{oneweightsection}. The endpoint estimate of Jessen,
Marcinkiewicz, and Zygmund  and its multilinear extension
(Theorem~\ref{BSMF}) are discussed in Section \ref{endpoint}. The
proof of the latter is presented in Section \ref{strongmaximalthm}.
Finally, the result on bilinear interpolation between distributional
estimates (Theorem~\ref{interpolation2}) and its application to
bilinear commutators are contained in Section \ref{interp22}.


\medskip

\noindent{\bf Acknowledgments:}
The authors would like to thank the referee for valuable remarks.

\section{Classical weighted theory for a general basis}\label{generalbasis}

In \cite{M} Muckenhoupt proved a fundamental result characterizing
all weights for which the Hardy--Littlewood maximal operator is
bounded. As it is nowadays well-known, the surprisingly simple
necessary and sufficient condition is the so-called
$A_{p}$ condition (cf. \eqref{multiapCubes} with $m=1$). A
different approach to this characterization was found by Jawerth
\cite{J} (see Theorem \ref{A} below) based on ideas of Sawyer
\cite{S}. Recently,  a simple and elegant proof of
this characterization,  which also yields the sharp bound in terms
of the $A_p$ constant of the weight, was given by  Lerner \cite{L}.

\subsection{ The maximal function for a general basis}

We start by introducing some notation. By a {\it basis} $\cb$ in
$\rn$ we mean a collection of open sets in $\rn$. We  say that $w$
is a {\it weight} associated with the basis $\mathcal B$ if $w$ is a
non-negative measurable function in $\rn$ such that $w (B) =
\int_{B} w(y)\, dy < \infty$ for each $B$ in $\mathcal B$.
$M_{\cb,w}$ is the corresponding maximal operator defined by
\[
 M_{{\mathcal B},w}(f)(x) =
\sup_{ \substack{ B\ni x \\ B\in {\mathcal B}  }} \frac{1}{w(B) }
\int_{B} |f(y)|\, w(y)dy
\]
for  $x \in \bigcup_{B \in \cb}B $ and $M_{\cb,w}f(x) = 0 $ for
$x\notin \bigcup_{B \in \cb} B$. If $w \equiv 1$, we simply write
$M_{\cb }f(x) $.


 Several 
important examples of bases arise by taking
$\cb= \cq$ the family of all open cubes in $\rn$ with sides parallel
to the axes, $\cb= \cd$ the family of all open dyadic cubes in
$\rn$,  and   $\cb = \crec$ the family of all open rectangles in
$\rn$ with sides parallel to the axes.   Other interesting examples 
are  given by bases of  open rectangles with sides parallel to the axes  and related side lengths.   We have for instance   the basis 
$\Re$  formed by all  rectangles in $\rr^3$ with side lengths are $s$, $t$, and $st$, for some $t,s>0$.  Similarly,  for a given parameter $N>1$,  we   consider the  the family $\Re_{N}$ of all  rectangles in $\rr^2$ with eccentricity $N$; namely those rectangles whose side lengths are $s$ and $Ns$ for some $s>0$.  This last basis  is associated 
with the so-called Nikodym maximal function. 


A weight $w$  associated with $\cb$ is said to satisfy the
$A_{p,\cb}$ condition, $1<p<\infty$, if
\begin{equation}\label{apb}
\sup_{B\in \cb} \left(\f 1 {|B|} \int _B  w \, dx \right)\,  \left( \f 1 {|B|}
\int _Bw^{1-p'}\, dx\right)^{\frac{p}{p'}} <\infty\, .
\end{equation}
In the limiting case $p = 1$ we say that $w$ satisfies the $A_{1, \mathcal B}$
if
$$
\left( \frac{1}{ |B| } \int_{B} w(y)\, dy \right) ess.sup_{B}
(w^{-1}) \le c
$$
for all $B \in \mathcal B$; this is equivalent to saying
$$
 M_{{\mathcal B}}w(x) \le c\, w (x)
$$
for almost all $x \in \rn$. It follows from these definitions and
H\"older's inequality that
$$
A_{p, \mathcal B} \subset   A_{q, \mathcal B}
$$
if $1 \le p \le q \le \infty$.  Then it is natural to define the
class $A_{\infty,\cb}$ by setting
$$
A_{\infty,\cb}= \bigcup_{p>1} A_{p,\cb}.
$$

One  reason  that this general framework is interesting to consider
 is the following theorem due to   Jawerth \cite{J} (see \cite{L} for a simpler proof).

\begin{theorem}\label{A}
Let $1  < p < \infty$. Suppose that $\mathcal B$ is a basis and that
$w$ is a weight associated with $\cb$, and set $\sigma = w^{1-p'}$. Then
$$
\left\{
\begin{array}{c}
 M_{\mathcal B}: \, L^{p}(w) \rightarrow L^{p}(w) \\
 M_{\mathcal B}: \, L^{p'}(\sigma) \rightarrow L^{p'}(\sigma)
\end{array}
\right.
$$
if and only if
$$
\left\{
\begin{array}{ll}
 w \in A_{p,\mathcal B}\\
 M_{{\mathcal B},w}:\, L^{p^{\prime}}(w) \rightarrow
L^{p^{\prime}}(w) \\
 M_{\mathcal B, \sigma}:\, L^{p}(\sigma) \rightarrow L^{p}(\sigma).
\end{array}
\right.
$$
\end{theorem}

Theorem \ref{A} includes Muckenhoupt's result, mentioned above, that
for  $1 < p < \infty$,
\[
 M_{\mathcal Q}:\, L^{p}(d\mu) \rightarrow L^{p}(d\mu)
\]
holds if and only if  $d\mu = w(y)dy$, with $w \in A_{p,\mathcal Q}$,
which is simply the classical $A_p$ condition.

A key fact is that the proof of Theorem \ref{A}  in \cite{J} or
\cite{L} completely avoids the (difficult) \lq \lq reverse H\"older inequality".

\subsection{Muckenhoupt basis}

Following \cite{P1}, we   use the following class of bases.

\begin{definition}
We say that $\cb$ is a {\it Muckenhoupt basis} if for any\, $1<p<\infty$
\begin{equation}\label{mbounded}
M_\cb:\, L^p(w)\to L^p(w)
\end{equation}
for any $w \in A_{p,\cb}$.
\end{definition}

It is shown in \cite{P1} that this definition is equivalent to the
following result:

\begin{theorem}\label{MuckbasChar}
$\cb$\, is a Muckenhoupt basis
 if and only if 
for any\, $1 < p < \infty$,
$$
M_{\cb,w}:\, L^{p}(w) \rightarrow L^{p}(w)
$$
whenever $w \in A_{\infty, \cb}$.
\end{theorem}

Most of the important bases are Muckenhoupt bases, and in particular
those mentioned above: $\cq, \cd$, $\crec$.  The fact that $\crec$
is a Muckenhoupt basis can be found in \cite{GCRdF}. The basis $\Re$
is also a Muckenhoupt basis as shown by R. Fefferman \cite{RF2}.

\subsection{ The multisublinear maximal operator for a general basis}

We are interested in extending some of the main results in
\cite{LOPTT} concerning the maximal operator in \eqref{first11} to
other bases. We introduce a multisublinear version of the maximal
operator  $\cm_\cb$  by setting
$$\cm_\cb ( f_1, \dots , f_m) (x)= \sup_{B\ni x} \prod_{i=1}^m \lf(\f 1{|B|}\int_B |f_i(y)|\,dy\r).$$

For a basis $\cb$  we define the multiple weight $A_{\vec P,\cb}$
condition as in \cite{LOPTT}:

\begin{definition}
Let  $1\le p_1, \dots , p_{m}<\infty$.
Given $\vec w=(w_1,\dots,w_m)$, set
\begin{equation}\label{oneweight}
\nu_{\vec w} = \prod_{i=1}^m  w_i^{p/p_i}  . 
\end{equation}
We say that the m-tuple of weights $\vec w$ satisfies the $A_{\vec
P,\cb}$ condition if
\begin{equation}\label{multiapBasisoneweight}
\sup_{B\in \cb}\Big(\frac{1}{|B|}\int_B\nu_{\vec w}(x)\,dx\Big)\,
\prod_{j=1}^m\Big(\frac{1}{|B|}\int_B w_j(x)^{1-p'_j}\,dx
\Big)^{\frac{p}{p'_j}}<\infty.
\end{equation}
When $p_j=1$, $\Big(\frac{1}{|B|}\int_B w_j^{1-p'_j}\Big)^{1/p'_j}$
is understood as $\displaystyle(\inf_Bw_j)^{-1}$.
We use \, $[\vec w\,]_{A_{\vec P, \cb}}$ \,to denote the quantity in \eqref{multiapBasisoneweight}.

\end{definition}

\subsection{The case of $m+1$ weights:  weak type estimates and the power bump condition}

In view of the $A_{\vec P,\cb}$ condition, it is natural to say that
the $(m+1)$-tuple of weights $(\nu, \vec w)$ satisfies the multiple
multilinear condition $A_{\vec P,\cb}$ if
\begin{equation}\label{multiapBasistwoweights}
\sup_{B\in \cb}\Big(\frac{1}{|B|}\int_B\nu(x)\,dx\Big)\,
\prod_{j=1}^m\Big(\frac{1}{|B|}\int_B w_j(x)^{1-p'_j}\,dx
\Big)^{\frac{p}{p'_j}}<\infty\, . 
\end{equation}
Note that here $\nu$ is not assumed to be the weight  $\nu_{\vec w}$
determined by $w$ in \eqref{oneweight}. Condition
\eqref{multiapBasistwoweights} is sufficient for the
characterization of the weak type  estimate  in the classical linear
case with $\cb=\cq$  and also in the multilinear version with the
same basis (cf. \cite{LOPTT}). It is unknown, however,  if it is
also sufficient for other bases. For example, it remains an open
problem whether this condition suffices for
 the basis $\crec$ 
even in the linear case.

The following definition gives a  stronger condition than
\eqref{multiapBasistwoweights} which is quite useful and is often
called {\it power bump condition}.

\begin{definition}\label{bumpcondition}
 We say that the $(m+1)$-tuple of weights $(\nu, \vec w)$  satisfies a {bump
$A_{\vec P,\cb}$ condition} if $\nu \in A_{\infty, \cb}$ and for some $r>1$,
\begin{equation}\label{bumpedap}
\sup_{B\in \cb} \f 1 {|B|} \int _B \nu(x)\, dx \, \prod_{j=1}^m
\left( \f 1 {|B|} \int _Bw_j^{(1-p'_j)r}\,
dx\right)^{\frac{p}{p_j'r}}<\infty\, .
\end{equation}
\end{definition}

This type of power bump condition appeared for the first time in the
work of  Neugebauer \cite{N1} for  $m=1$ and $\cb={\mathcal Q}$, but
with an extra power bump in the weight $\nu$. P\'erez  \cite{P3}
removed the power from  the weight $\nu$ and replaced the power bump
in $w$ by  a {\it logarithmic bump } or a more general type of bump.
Such power bump conditions were then used in \cite{P4} and  \cite{P5},
and in the work of  Cruz-Uribe et al \cite{CMP} to prove very sharp two-weighted estimates for
classical operators.
 For a general $m$ and the basis $\cb=\cq$,
Moen \cite[Theorem~2.8]{moen} obtained that
$\mathcal M_\cq:\,  L^{p_1}(w_1) \times 
\cdots \times L^{p_m}(w_m) \to L^{p}(\nu)$
provided that $1 < p_1, . . . , p_m < \infty$, 
$\f1 p=\f 1 {p_1} + \cdots +\f 1 {p_m}$,
and $(\nu, \vec w)$ satisfy the power bump condition \eqref{bumpedap} for some $r>1$.

We have the following  result concerning $(m+1)$-tuples of weights.

\begin{theorem} \label{twoweights}
Let $\cb$ be a Muckenhoupt  basis and let $\vec P = (p_1,\dots,p_m)$
with $1<p_1,\dots,p_m<\infty$ and $\f1 p=\f 1 {p_1} + \cdots +\f 1 {p_m} $. Let $(\nu, \vec w)$ satisfy the power bump condition
\eqref{bumpedap} for some $r>1$, then
$$\cm_\cb: \,  L^{p_1}(w_1)
\times \cdots \times L^{p_m}(w_m) \to L^{p,\infty}(\nu).$$
\end{theorem}

When applied to the basis $\crec$ and $\nu=\nu_{\vec
w}$, the above theorem gives the following characterization.

\begin{corollary} \label{weakoneweight}

Let $\vec P = (p_1,\dots,p_m)$ with  $1<p_1,\dots,p_m<\infty$ and
$\f1 p=\f 1 {p_1} + \cdots +\f 1 {p_m}$ and let $\vec w$ be an $m$-tuple of
weights. Then
$$\cm_{\crec}: L^{p_1}(w_1) \times \cdots \times L^{p_m}(w_m) \to L^{p,\infty}(\nu_{\vec w})
\qq \mbox{if and only if} \qq  \vec w \in A_{\vec P, \crec}.$$
\end{corollary}

The proof of the corollary follows known arguments.  The necessity of
$A_{\vec P, \crec}$ is quite standard and we omit the details. For
the sufficiency we first observe that as in \cite[Theorem
3.6]{LOPTT}, the  vector condition \eqref{multiapBasisoneweight}
implies that $\nu_{\vec w}$ is in the (linear) $A_{mp, _{\crec}}$ class
and $w_i^{1-p_i'}$ is in the
linear $A_{mp_i', _{\crec}}$ class. In fact, the arguments used in
\cite{LOPTT} rely only on the use of H\"older's inequality on the
sets were the averages involved in the various $A_p$ conditions take
place, so the arguments would also apply to any other differentiating
basis. Using the reverse H\"older inequality property of the basis
of rectangles (see the book by Garc\'ia-Cuerva and Rubio de Francia \cite{GCRdF} p.\,458)
we can now ``bump" the weights $w_i^{1-p_i'}$ and apply  Theorem \ref{twoweights}.

We observe that the case $p_1=\cdots=p_m=1$ is excluded in the statements of Theorem~\ref{twoweights} and
Corollary~\ref{weakoneweight}.  The problem for  this endpoint case
remains open except for the unweighted case, which will be
considered in Section \ref{endpoint} below.

Finally, we remark that  the proof
of the Theorem \ref{twoweights} given in Section \ref{two-weight}, yields
that the $A_\infty$ condition assumed on $\nu$ can be replaced by
the weaker condition (A) given in Definition \ref{conditionA}.

\subsection{The case of $m$-weights:  strong bounds for the strong maximal
function.} \label{strmaxstrbound}

\

A characterization of the strong type bounds for the strong maximal function is possible in this case and we obtain the following.%
\begin{theorem} \label{1wghtStrChar}
Let $1<p_j<\infty$, $j=1,\dots,m$ \,and\,
$\frac{1}{p}=\frac{1}{p_1}+\dots+\frac{1}{p_m}$.\, Then
$$
\cm_{\crec}: \,   L^{p_1}(w_1) \times \cdots \times L^{p_m}(w_m) \to
L^{p}(\nu_{\vec w}) \qq \mbox{if and only if} \qq \vec w \in A_{\vec
P, \crec}.
$$
\end{theorem}

\section{The unweighted endpoint estimates for the  multilinear strong maximal function}  \label{endpoint}

In this section we focus attention on unweighted endpoint properties
of the strong multilinear maximal function $\cm_{\crec}$.

As mentioned in the introduction,   \eqref{e3} and   H\"older's
inequality   yield  that $\cm_{\crec}$ is bounded from $L^{p_1}(\rn)
\times  \dots \times L^{p_m}(\rn)$ to $L^p(\rn)$ whenever $1<p_1,
\dots, p_m \le\nf$ and  $\f1p=\f1{p_1}+ \dots +\f1{p_m}$.  This
argument does not apply to the endpoint case $L^1\times \dots \times
L^1$ since the operator $M_{\crec}$ is not of weak type $(1,1)$.

 We have also   mentioned  a substitute to the weak type $(1,1)$ endpoint estimate
for  $M_{\crec}$  obtained by Jessen, Marcinkiewicz and Zygmund
\cite{jmz},   who showed that  that for all $f $ on $\rn$   
\begin{equation}\label{e1}
|\{x\in\rn:\, M_{\crec}(f)(x)>\la\}| \le C_n\int_{\rn}
\Phi_n\lf(\f{|f(x)|}{\la}\r)\,dx,\,  
\end{equation}
 where  constant $C_n$ is a constant depending
only on the dimension $n$ and 
for $ t>0$%
$$
\Phi_n(t)=t(1+(\log^+t)^{n-1}) \approx  t(\log(e+t))^{n-1}.
 $$
Later, C\'orboda and Fefferman \cite{cf} gave a geometric proof of
\eqref{e1} answering a question formulated by Zygmund.

These kind of distributional estimates have also appeared in other
works in the literature, for instance in \cite{P6} for $M^2$ (the
composition of the maximal function with itself)  and for
commutators. They are interesting because they provide good endpoint
estimates for purposes of interpolation.

As in the linear case, we should not expect weak type estimates when
 $p_i=1$ for all $i$, but rather a distributional estimate involving the
function $\Phi_n$.  Moreover, we will need to consider compositions
of the function $\Phi_n$ with itself.  In general, for a given $m\in\nn$ and
$\Phi$ increasing, we set
$$\Phi^{(m)}:=
\overbrace{\Phi\circ \Phi \circ\cdots\circ\Phi}^{m \mbox{ times
}}.$$
Then $\Phi^{(m)}$ is an increasing function and is also increasing
with respect to $m$. In the special case of $\Phi=\Phi_n$ we will
use the following estimate: there exists a positive constant
$C_{n,m}$ which depends on $n$ and $m$ such that for $t>0$,
\begin{eqnarray}\label{Phinm}
\Phi_n^{(m)}(t)\le C_{n, m} t [\log(e+t)]^{m(n-1)}.
\end{eqnarray}

\begin{theorem}\label{BSMF} There exists a positive
constant $C$ depending only on $m$ and $n$ such that for all $
\la>0$,
$$
\q\lf|\lf\{x\in\rn:\, \cm_{\crec}(\vec f\,)(x)>\la^m\r\}\r| \le C
\lf\{\prod_{i=1}^m\int_{\rn} \Phi_n^{(m)}
\lf(\f{|f_i(x)|}{\la}\r)\,dx\r\}^{1/m}
$$
 for all  $ f_i $ on $\rn$ 
  and  for all $i=1,\dots , m$.
Furthermore,  the theorem is sharp in the sense that we cannot
replace $\Phi_n^{(m)}$ by $\Phi_n^{(k)}$ for $k\leq m-1$.
\end{theorem}

We present the proof of this result in Section
\ref{strongmaximalthm}.

\section{Proof of the weak type estimate in the $(m+1)$-weight case} \label{two-weight}

\begin{proof}[ \textbf{Proof of Theorem \ref{twoweights}.} ]
By homogeneity it is enough to show that
\begin{equation}\label{weaktype}
\nu \left( \{ x\in \rn: \cm_\cb(\vec f \, ) (x) > 1\} \right)^{1/p}
 \lesssim \prod_{i=1}^m \left( \int_{\rn} |f_i|^{p_i}w_i \,dx \right)^{1/p_i}.
 \end{equation}
Moreover  it is enough to prove \eqref{weaktype} uniformly for any
compact set $K$ contained in the set $\{ x\in \rn: \cm_\cb(\vec f ) (x) >
1\}$.

Given such a compact set $K$ we can cover it with a finite
collection of sets $\{B_j\}_{j=1}^N$ in $\cb$ such that
\begin{equation*}
 \prod_{i=1}^m   \f 1 {|B_j|} \int _{B_j} |f_i| \, dy >1
 \end{equation*}
for all $j$. As in \cite{P2} we follow   a well-known selection
procedure (see for instance \cite{GCRdF} p.\,463,  \cite{C}, or
\cite{J}). Then we   extract a subfamily $\{\widetilde
B_j\}_{j=1}^{M}$, selected in such a way that $\widetilde B_1 =B_1$,
\begin{equation} \label{small}
\left| \widetilde B_k \bigcap \bigcup_{j=1}^{k-1} \widetilde B_j
\right| < |\widetilde B_k|/2,
\end{equation}
and if $B_l$ is not in the subfamily  $\{\widetilde B_j\}_{j=1}^{M}$
then
\begin{equation} \label{large}
\left| B_l \bigcap \bigcup_{j=1}^{M} \widetilde B_j \right| \geq
|B_l|/2.
\end{equation}
Note that if we define
$$
F_k = \widetilde B_k \bigcap \bigcup_{j=1}^{k-1} \widetilde B_j,$$
then the set
$$
E_k= \widetilde B_k \setminus F_k$$
 satisfies
\begin{equation}\label{comparable}
|E_k| \approx |\widetilde B_k |,
\end{equation}
for all $k$ and the sets  $\{E_j\}_{j=1}^{M}$ are pairwise disjoint.

We now claim that
\begin{equation}\label{claimstrong}
\bigcup_{j=1}^N B_j \subset \left\{ x \in \rn: \cm_\cb \left( \vec
\chi_{\bigcup_{j=1}^M \widetilde B_j}\right)(x) \geq 2^{-m}
\right\},
\end{equation}
where
$$\vec \chi_{\bigcup_{j=1}^M \widetilde B_j} = (\chi_{\bigcup_{j=1}^M
\widetilde B_j} ,\dots, \chi_{\bigcup_{j=1}^M \widetilde B_j} ).$$
Assume the claim for a moment and set
$$
\left\{ x \in \rn: \cm_\cb \left( \vec \chi_{\bigcup_{j=1}^M
\widetilde B_j}\right)(x) \geq 2^{-m} \right\} = G .
$$
Since  $\nu  \in A_{\infty,\cb}$, $\nu$ is in   $A_{mq,\cb}$
  for some  $q>1/m$. We have
\begin{eqnarray*}
\nu (G) & \leq &
  2^{mq}  \int_{\rn}  \left( \prod_{i=1}^m M_\cb (\chi_{\bigcup_{j=1}^M \widetilde B_j})\right)^q\nu(x) \, dx\\
& \lesssim &
 \int_{\rn}   \left( M_\cb (\chi_{\bigcup_{j=1}^M \widetilde B_j})\right)^{mq}\nu(x) \, dx\\
& \lesssim &
 \int_{\rn}   \left( \chi_{\bigcup_{j=1}^M \widetilde B_j}\right)^{mq}\nu(x) \, dx\\
& \lesssim & \nu \bigg(\bigcup_{j=1}^M \widetilde B_j\bigg),
 \end{eqnarray*}
and from \eqref{claimstrong} it follows that
$$
\nu (K) \leq \nu \bigg(\bigcup_{j=1}^M  B_j\bigg) \lesssim\nu \bigg(\bigcup_{j=1}^M \widetilde B_j\bigg) \lesssim \sum_{j=1}^M  \,    \nu (\widetilde B_j)
\, .
$$
 Now using the bump  condition on the weights with some $r>1$, we can estimate the measure of the   expression on the right above
as follows: 
\begin{eqnarray*}
& &  \sum_{j=1}^M  \,    \nu (\widetilde B_j)\\
& \le  & \sum_{j=1}^M  \, \left(  \prod_{i=1}^m   \f 1 {|\widetilde
B_j|} \int _{\widetilde B_j} |f_i| \, dx \right)^p\, \,
 \nu (\widetilde B_j)\\
& \lesssim & \sum_{j=1}^M  \prod_{i=1}^m   \left(  \f 1 {|\widetilde
B_j|}
 \int _{\widetilde B_j} |f_i|^{(p_i'r)'}  w_i^{(p_i'r)'/p_i} \, dx \right)^{p/(p_i'r)'}
  \prod_{i=1}^m   \left(  \f 1 {|\widetilde B_j|}
 \int _{\widetilde B_j}  w_i^{-p_i'r/p_i} \, dx \right)^{p/p_i'r}
 \nu (\widetilde B_j)\\
 & \lesssim &
\sum_{j=1}^M  \prod_{i=1}^m   \left(  \f 1 {|\widetilde B_j|}
 \int _{\widetilde B_j} |f_i|^{(p_i'r)'}  w_i^{(p_i'r)'/p_i} \, dx \right)^{p/(p_i'r)'}
 |\widetilde B_j|, 
 \end{eqnarray*}
 in view of \eqref{bumpedap}.
Finally using \eqref{comparable}, H\"older's inequality with $\sum
\frac{p}{p_i} =1$,   the fact that $p_i/(p_i'r)'>1$,  we   continue
the preceding  sequence of inequalities as follows:
\begin{eqnarray*}
& \lesssim & \sum_{j=1}^M  \int _{ E_j}   \prod_{i=1}^m
\left( M_\cb \left( |f_i|^{(p_i'r)'}  w_i^{(p_i'r)'/p_i} \right) \right)^{p/(p_i'r)'} \, dx \\
 & \lesssim &
 \int _{ \rn}   \prod_{i=1}^m
\left( M_\cb \left( |f_i|^{(p_i'r)'}  w_i^{(p_i'r)'/p_i}\right) \right)^{p/(p_i'r)'} \, dx\\
& \lesssim & \prod_{i=1}^m   \left(  \int _{ \rn}
\left( M_\cb \left( |f_i|^{(p_i'r)'}  w_i^{(p_i'r)'/p_i}\right) \right)^{p_i/(p_i'r)'} \, dx \right)^{p/p_i}\\
& \lesssim & \prod_{i=1}^m   \left(  \int _{ \rn} |f_i|^{p_i} w_i\,
dx \right)^{p/p_i},
\end{eqnarray*}
which gives the desired weak type estimate.

It only remains to verify \eqref{claimstrong}. To do so, fix $x$ in
$\bigcup_j B_j$. If $x$ is in one of the sets  $\widetilde B_k$,
then  $x$ is in $G$ because
$$
\cm_\cb \left( \vec \chi_{\bigcup_{j=1}^M \widetilde B_j}\right)(x)
\geq \prod_{i=1}^m \f 1 {|\widetilde B_k|} \int_{\widetilde B_k}
\chi_{\widetilde B_k}(y)\,dy = 1.
$$
On the other hand, if $x \notin \bigcup_j \widetilde B_j$, then $x$
is in some $B_k$ satisfying \eqref{large}. It follows then that
$$
\cm_\cb \left( \vec \chi_{\bigcup_{j=1}^M \widetilde B_j}\right)(x)
\geq \prod_{i=1}^m \f 1 {|B_k|} \int_{ B_k} \chi_{\bigcup_{j=1}^M
\widetilde B_j}(y)\,dy = \prod_{i=1}^m \f 1 {|B_k|}  \left|B_k
\bigcap \bigcup_{j=1}^M \widetilde B_j \right| \geq 2^{-m},
$$
and so $x$ is also in $G$. The proof of the theorem is now complete.
\end{proof}

\section{The strong maximal function and the $m$-weight case} \label{oneweightsection}

The purpose of this section is prove Theorem \ref{1wghtStrChar}
concerning the multilinear strong maximal operator $\cm_{\crec}$
related to the basis $\crec$.  To do this we need a special case of
a lemma from \cite{J}  and some
 additional 
definitions for general
basis.  The first definition concerning the concept of {\it
$\alpha$-scattered families}  is of geometric nature  and  plays an
important role in this context. It has been considered in the works
\cite{JT} and \cite{J} and implicitly in \cite{C} and \cite{cf}.

\begin{definition}

Let $\cb$ be a basis and let $0<\alpha<1$.  A finite sequence
$\{\tilde{A_i}\}_{i=1}^{M} \subset\cb$ of sets of finite
$dx$-measure is called $\alpha$--scattered with respect to the
Lebesgue  measure if
$$
 \left|\tilde{A_i}\bigcap \bigcup_{s<i}\tilde{A_s}\right|\leq \alpha |\tilde{A_i}|
$$
for all $1<i\leq M$.
\end{definition}

Next, we will be considering an important large class of weights.

\begin{definition} \label{conditionA}
Let $\cb$ be a basis and let $w$ be a weight associated to this
basis.  We say that $w$ satisfies condition $(A)$ if there are
constants $0 < \lambda < 1$, $0 <   c(\lambda) < \infty$ such that
for all measurable sets $E$ we have
$$
w \left( \{ x \in \rrr^{n} : M_{\cb}(\chi_{E})(x)
> \lambda \} \right) \le c(\la)\, w(E). \leqno(A)
$$
\end{definition}

This class of weights was also considered in \cite{JT} and further
exploited in \cite{J}. However, for the Lebesgue measure
it goes back to the work in \cite{C} and has been recently
considered again by Hagelstein and  Stokolos in their  articles
\cite{HS1}, \cite{HS2}.

One of the reasons  that   condition (A) is interesting is the fact
that it is weaker than the $A_{\infty, \cb}$ condition whenever the
basis $\cb$ is a Muckenhoupt basis. Indeed, if $w\in A_{\infty,
\cb}$,  then  $w\in A_{p, \cb}$ for some $p>1$ large enough. Then
for any measurable set $E$ we have
$$
w\left( \{ x \in \rrr^{n} : M_{\cb}(\chi_{E})(x)
> \lambda \} \right)
\leq \frac{1}{\la^p}
\int M_{\cb}(\chi_{E})(x)^p\, w(x)dx\\
\le c({\la})\, w(E),
$$
since $w\in A_{p, \cb}$ and the basis $\cb$ is a Muckenhoupt basis.
It follows that
 $w$ satisfies condition (A).

\begin{lemma}\label{scatteredproperty}
Let $\cb$ be a basis and let $w$ be a weight associated to this
basis. Suppose further that $w$ satisfies condition $(A)$   for some
$0<\la<1$  and $0<c(\la)<\infty$. Then  given any finite
sequence $\{A_i\}_{i=1}^{M}$ of sets  $A_i\in \cb$,\\

1) we can find a subsequence $\{\tilde{A_i}\}_{i\in I}$ of
$\{A_i\}_{i=1}^{M}$ which is  $\la$-scattered with respect to the Lebesgue measure; \\

2) $\tilde{A_i}= A_i, \,  i\in I$; \\

3) for any $1\leq i<j\leq M+1$
\begin{equation}\label{thirdscattproperty}
w\Big(\bigcup_{s<j}A_s\Big)\leq   c(\la)
\,\Big[w\Big(\bigcup_{s<i}A_s\Big) + w\Big(\bigcup_{i\leq
s<j}\tilde{A_{s}}\Big)   \Big],
\end{equation}
where $\tilde{A_{s}}=\emptyset$ when $s\notin I$.
\end{lemma}

The proof of this result is adapted from p. 370 Lemma 1.5 \cite{J}.

\begin{proof}
We let $\tilde{A_1}= A_1$\, be the first element of the subsequence.
Suppose now that\, $\tilde{A_1}, \tilde{A_{i_2}}, \cdots,
\tilde{A_{i_{l-1}}}$\,  have been already selected. Then
$\tilde{A_{i_l}}=A_{i_l}$\, where\, $A_{i_l}$ is the first element
$A$ of $\{A_i\}_{i=1}^{M}$  after $A_{i_{l-1}}$ with the property
$$ \left|A\bigcap \bigcup_{s\leq i_{l-1}}\tilde{A_s}\right|\leq \la |A|.$$
We continue this way until $\{A_i\}_{i=1}^{M}$ is exhausted. For the
subsequence\, $\{\tilde{A_i}\}_{i\in I}$\, obtained in this way we
clearly have\, $\tilde{A_i}=A_i, \, i\in I$. We claim that
\begin{equation}\label{ggegge}
\bigcup_{s<j}A_s \subseteq \{ x \in \rrr^{n} :
M_{\cb}(\chi_{E})(x)
> \lambda \} 
\end{equation}
where \,$E=\bigcup_{s<j}\tilde{A_s}$. Since $w$ satisfies condition
$(A)$ it is easy to deduce immediately \eqref{thirdscattproperty}.
To prove the claim we first observe that\,
$\bigcup_{s<j}\tilde{A_s}$ is trivially contained in the set   on the
right in \eqref{ggegge}. 
On the other hand, if $A\in \{A_s\}_{s<j}\setminus
\{\tilde{A_s}\}_{s<j}$, then at some index before $j$, A was
discarded in the above selection procedure. But this means that\, $
\left|A\bigcap \bigcup_{s<j_0}\tilde{A_s}\right|> \la |A|$\, for
some \,$j_0\leq j$ \,and hence, $A \subseteq \{ x \in \rrr^{n} :
M_{\cb}(\chi_{E})(x)
> \lambda \}$\, for\,$E=\bigcup_{s<j}\tilde{A_s}$ verifying the
claim. This finishes the proof of the lemma.
\end{proof}

\begin{proof}[ \textbf{Proof of Theorem \ref{1wghtStrChar}.} ]

For a vector $\vec f $ of bounded functions with compact support
consider again   the strong multilinear maximal function
$$
\cm_{\crec} (\vec f \,) (x)= \sup_{R\ni x} \prod_{\al=1}^m \lf(\f
1{|R|}\int_R |f_{\al}(y)|\,dy\r) \, ,
$$
where the supremum is taken over all  rectangles $R$  in $\rn$ with sides parallel to the axes.
We will use an argument that combines ideas from \cite{LOPTT} (second
proof of Theorem 3.7), \cite{J} and \cite{P3}.

Let $N>0$ be a   large  integer.  We  will prove the required
estimate for the quantity
\begin{equation}\label{qqq111}
\int_{2^{-N}<\cm_{\crec}(\vec f\,)\leq 2^{N+1}} \cm_{\crec}(\vec
f\,)(x)^{p}\, \nu_{\vec w}(x)\, dx
\end{equation}
with a bound independent of  $N$. For each integer $k$, $|k|\leq N$,
we   find a compact set
$$
K_k  \subset \{\cm_{\crec}(\vec f\,)>2^k \}
$$
satisfying
$$
\nu_{\vec w}(K_k) \leq \nu_{\vec w}(\{\cm_{\crec}(\vec f\,)>2^k \})\leq 2 \, \nu_{\vec w}(K_k)
$$
and a finite sequence $b_k=\{B_j^k\}_{j\geq 1}$ of sets $B_j^k \in
\crec$ with
\begin{equation}\label{bigger2^k}
 \prod_{\al=1}^m   \f 1 {|B_j^k|} \int _{B_j^k} |f_\al (y)| \, dy >2^k
 \qqq j\geq 1.
\end{equation}
 We  use a selection procedure as in
 \cite[p.\,372]{J}. 
For convenience we set $b_k=\emptyset$ if $|k|>N$
and
\begin{equation}\label{cases}
\Omega_k=
\begin{cases}
\bigcup_{s\geq k}\bigcup_{j}B_j^s  &\textup{when }\q |k|\leq N,
\\
\emptyset \qqqq  \qq &\textup{when }\q |k|>N.
\end{cases}
\end{equation}
Observe that these sets are decreasing in $k$, i.e.,
$\Omega_{k+1}\subset  \Omega_k$.

We now distribute the sets in $\bigcup_k b_k$ over $\mu$ sequences
$\{A_i(l)\}_{i\geq 1}$, $0\leq l\leq \mu-1 $, where $\mu$ will be
chosen momentarily to be an appropriately large natural number. Set
$i_0(0)=1$. In the first $i_1(0)-i_0(0)$ entries
of $\{A_i(0)\}_{i\geq 1}$, i.e., for
$$
 i_{0}(0)\leq i<i_{1}(0),
$$
we place the    elements of the sequence $b_N=\{B_j^N\}_{j\geq 1}$
  in the order indicated by the index $j$. For the next
  $i_{2}(0)-i_{1}(0)$ entries of $\{A_i(0)\}_{i\geq 1}$, i.e., for
$$
 i_{1}(0)\leq i<i_{2}(0),
$$
we place   the elements of the sequence $b_{N-\mu}$. We continue in this way until
we  reach the first integer $m_0$ such that $N-m_0\mu\geq -N$, when
we stop. For indices $i$  satisfying
$$
i_{m_0}(0)\leq i<i_{m_0+1}(0),
$$
we place  in the sequence   $\{A_i(0)\}_{i\geq 1}$ the    elements of
$b_{N-m_0\mu}$. The sequences
$\{A_i(l)\}_{i\geq 1}$, $1\leq l\leq \mu-1,$ are defined similarly,
starting from $b_{N-l}$ and using the families
$b_{N-l-s\mu}$, $s=0,1,\cdots, m_l$,  where $m_l$ is chosen
so that $N-l-m_l\mu\geq -N$. 

 Since $\nu_{\vec w} \in A_{\infty, _{\crec}}$,
$\nu_{\vec w}$ satisfies condition (A) by the remark made after
Definition \ref{conditionA} and we can apply Lemma
\ref{scatteredproperty} to each $\{A_i(l)\}_{i\geq 1}$ for some
fixed $0<\la<1$. Then we obtain sequences
$$
\{\tilde{A}_i(l)\}_{i\geq 1} \subset \{A_i(l)\}_{i\geq 1} \, , \qqq 0\leq
l\leq \mu-1,
$$
which are $\la$-scattered with respect to the
Lebesgue measure. In view of the definition of the set $\Omega_k$ and
the construction of the families $\{A_i(l)\}_{i\geq 1}$, we can use
 Assertion 3) 
of Lemma \ref{scatteredproperty} to obtain
$$  \nu_{\vec w}(\Omega_k) \leq  c\Bigg[ \nu_{\vec w}(\Omega_{k+\mu}) +
\nu_{\vec w}\lf( \bigcup_{ i_{m_l}(l)\leq i<i_{m_l+1}(l)
}\tilde{A}_i(l)\r) \Bigg] \leq c\, \nu_{\vec w}(\Omega_{k+\mu}) +
c\! \sum_{i=i_{m_l}(l)}^{i_{m_l+1}(l)-1} \nu_{\vec w}(\tilde{A}_i(l))
$$
if $k=N-l-m_l\mu$.  It will be enough to consider these indices $k$
because the sets $\Omega_k$ are decreasing.

Now, all the sets $\{\tilde{A}_i(l)\}_{i=i_m(l)}^{i_{m+1}(l)-1}$
belong to $b_k$ with  $k= N-l-m_l\mu $  and therefore
$$
\prod_{\al=1}^m   \f 1 {|\tilde{A}_i(l)|} \int _{\tilde{A}_i(l)}
|f_\al (x)| \, dx
>2^k.
$$
 It now readily follows that
$$
\int_{2^{-N}<\cm_{\crec}(\vec f\,)\leq 2^{N+1}} \cm_{\crec}(\vec
f\,)(x)^{p}\, \nu_{\vec w}(x)\, dx \leq 2^p \sum_k2^{kp} \nu_{\vec
w}(\Omega_k)
$$
and then
\begin{eqnarray*}
&& \sum_k2^{kp} \nu_{\vec w}(\Omega_k) \leq\  c \sum_k2^{kp}
\nu_{\vec w}(\Omega_{k+\mu})\,+c\,
 \sum_{l=0}^{\mu-1} \sum_{i\in I(l)} \nu_{\vec w}(\tilde{A}_i(l))
\left[\prod_{\al=1}^m   \f 1 {|\tilde{A}_i(l)|} \int
_{\tilde{A}_i(l)} |f_\al |   dx\right]^p  \\
&=&  c \,2^{-p\mu}\sum_k 2^{kp} \nu_{\vec w}(\Omega_{k})\,+c\,
 \sum_{l=0}^{\mu-1} \sum_{i\in I(l)} \nu_{\vec w}(\tilde{A}_i(l))
\bigg(\prod_{\al=1}^m   \f 1 {|\tilde{A}_i(l)|} \int
_{\tilde{A}_i(l)} |f_\al (x)| \, dx\bigg)^p.
\end{eqnarray*}
If we choose $\mu$ so large that $c\, 2^{ -\mu p}\leq\f12$ and since
everything involved is finite the first term on the right hand side
can be subtracted from the left hand side. This yields
$$
\int_{2^{-N}<\cm_{\crec}(\vec f\,)\leq 2^{N+1}}
\!\!\!\! \! \cm_{\crec}(\vec
f\,)^{p}\, \nu_{\vec w} \, dx \leq  2^{p+1}c\, \sum_{l=0}^{\mu-1}
\sum_{i\in I(l)} \nu_{\vec w}(\tilde{A}_i(l)) \left(\prod_{\al=1}^m
\f 1 {|\tilde{A}_i(l)|} \int _{\tilde{A}_i(l)} |f_\al| \,
dx\right)^p.
$$
We now proceed as in the proof of Theorem \ref{twoweights}: for each
$\al$ we use   H\"older's inequality  with exponents $p_{\al}'r$ and
$(p_{\al}'r)'$ to bound the previous expression on the right by
 \begin{eqnarray}\label{cont1}
&& 2^{p+1}c\, \sum_{l=0}^{\mu-1} \sum_{i\in I(l)}  \, \left(
\prod_{\al=1}^m \f 1 {|\tilde{A}_i(l)|} \int _{\tilde{A}_i(l)}
|f_{\al}| \, dx \right)^p\, \,
 \nu_{\vec w}(\tilde{A}_i(l))   \notag \\
 &&\quad\leq \ 2^{p+1}c\,
 \sum_{l=0}^{\mu-1} \sum_{i\in I(l)}    \prod_{\al=1}^m   \left(  \f 1 {|\tilde{A}_i(l)|}
 \int _{\tilde{A}_i(l)}  |f_{\alpha}|^{(p_{\al}'r)'}  w_{\alpha}^{ \f{(p_{\al}'r)'}{p_{\al}}} \! dx \right)^{\f{p}{(p_{\al}'r)'}}  \notag
  \\ &&\qqq\times
 \left(  \f 1 {|\tilde{A}_i(l)|}
 \int _{\tilde{A}_i(l)}  \! w_{\alpha}^{\f{-p_{\al}'r}{p_{\al}}}   dx \right)^{ \f{p}{p_{\al}'r}}\!
 \nu_{\vec w}(\tilde{A}_i(l))\, .
 \end{eqnarray}


Recall that  each $\si_{\al}=w_{\al}^{1-p_{\al}'}$ satisfies the
$A_{\infty,\crec}$ condition and hence by the reverse H\"older
inequality property of the basis $\crec$ (see \cite[p.\,458]{GCRdF})
there are constants $r_{\al}, c_{\al}>1$ such that
$$
\left(  \f 1 {|R|} \int _{R}  \si_{\al}^{r_{\al}}\,dx
\right)^{1/r_{\al}}  \leq c_{\al}\, \f 1 {|R|} \int _{R}
\si_{\al}\,dx  \qqq R\in \crec.
$$
Since $\vec w \in A_{\vec P, \crec}$ we can therefore bound
\eqref{cont1} by
\begin{equation}\label{cont2}
  C\, \sum_{l=0}^{\mu-1} \sum_{i\in I(l)}  \,  \prod_{\al=1}^m
\left( \f 1 {|\tilde{A}_i(l)|}
 \int _{\tilde{A}_i(l)} |f_{\al}|^{(p_{\al}'r)'}  w_{\al}^{\f{(p_{\al}'r)'}{p_{\al}}} \, dx \right)^{\f{p}{(p_{\al}'r)'}}
 |\tilde{A}_i(l)|\, .
\end{equation}
%
For each $l$ we let,
$$E_1(l)= \tilde{A}_1(l) \q \& \q E_i(l)= \tilde{A}_i(l) \setminus \bigcup_{s<i}
\tilde{A}_s(l) \qqq i>1.
$$
and we recall that the sequences $a(l)=\{\tilde{A}_i(l)\}_{i\in
I(l)}$ are $\la$--scattered with respect to the Lebesgue measure, hence %
$$
|\tilde{A}_i(l)| \leq \frac{1}{1-\la }|E_i(l)| \qqq i>1.
$$
Then we have the following estimate for \eqref{cont2}
\begin{equation}\label{cont3}
\f{C}{1-   \la  }    \sum_{l=0}^{\mu-1} \sum_{i\in I(l)}  \,
\prod_{\al=1}^m \left( \f 1 {|\tilde{A}_i(l)|}
 \int _{\tilde{A}_i(l)} |f_{\al}|^{(p_{\al}'r)'}  w_{\al}^{\f{(p_{\al}'r)'}{p_{\al}}} \, dx \right)^{\f{p}{(p_{\al}'r)'}}
|E_i(l)|.
\end{equation}
Now, since the family $\{E_i(l)\}_{i,l}$ consists of pairwise
disjoint sets and since
$$
\sum_{\al=1}^m \frac{p}{p_{\al}} =1\, ,
$$
using H\"older's inequality, we estimate \eqref{cont3} by a constant
multiple of
\begin{eqnarray*}
& & \sum_{l=0}^{\mu-1} \sum_{i\in I(l)}  \,
 \int _{E_i(l)}   \prod_{\al=1}^m  \left( M_{\crec} \left( |f_\al|^{(p_{\al}'r)'}  w_{\al}^{\f{(p_{\al}'r)'}{p_{\al}}} \right) \right)^{\f{p}{(p_{\al}'r)'}} \, dx \\
  & \leq  &\ c_{\mu}
 \int _{ \rn}   \prod_{\al=1}^m
\left( M_{\crec} \left( |f_{\al}|^{(p_{\al}'r)'}  w_{\al}^{\f{(p_{\al}'r)'}{p_{\al}}}\right) \right)^{\f{p}{(p_{\al}'r)'}} \, dx\\
& \lesssim &\ \prod_{\al=1}^m   \left(  \int _{ \rn}
\left( M_{\crec} \left( |f_{\al}|^{(p_{\al}'r)'}  w_{\al}^{\f{(p_{\al}'r)'}{p_{\al}}}\right) \right)^{\f{p_{\al}}{(p_{\al}'r)'}} \, dx \right)^{\f{p}{p_{\al}}}\\
& \lesssim &\ \prod_{\al=1}^m   \left(  \int _{ \rn}
|f_{\al}|^{p_{\al}} \, w_{\al}\,dx \right)^{\f{p}{p_{\al}}},
 \end{eqnarray*}
since $p_{\al}/(p_{\al}'r)'>1$, which gives the desired strong-type
estimate for \eqref{qqq111}. Letting $N\to \nf$ yields the claimed
assertion of the theorem.
\end{proof}

\section{Proof of the unweighted endpoint estimate for $\cm_{\crec}$} \label{strongmaximalthm}

In this section we   prove Theorem \ref{BSMF}. We begin by setting
 some notation and by proving several important ingredients required in the
proof.

\subsection{Orlicz spaces and normalized measures}\label{orlicz}

We need some basic facts from the theory of Orlicz spaces that we
  state without proof.  We refer to the book of Rao and Ren  \cite{RR} for the proofs and
  more information on Orlicz spaces.
  For  a lively
exposition of these spaces the reader may also consult the recent
book by Wilson \cite{W}.

 A Young function is a
continuous, convex, increasing function
$\Phi:[0,\infty)\to[0,\infty)$ with $\Phi(0)=0$ and such that
$\Phi(t)\to \infty$ as $t\rightarrow\infty$. The properties of
$\Phi$ easily imply that for $0<\ez<1$ and $ t\geq 0$
\begin{equation}\label{property1}
\Phi(\ez\, t) \leq \ez\,\Phi(t)\, .
\end{equation}
The $\Phi$-norm of a function $f$ over a set $E$ with finite measure
is defined by
\begin{equation}\label{property0}
\|f\|_{\Phi, E}=
\inf \lf\{\la>0\,:\, \frac{1}{|E|}\int_E  \Phi \left
(\frac{|f(x)|}{\la }\right )dx\leq 1\r\}.
\end{equation}
We will use the fact that
\begin{equation}\label{property2}
\|f\|_{\Phi, E} >1 \q \mbox{if and only if} \q \frac{1}{|E|}\int_E
\Phi \left (|f(x)|\right )dx > 1.
\end{equation}

 Associated with each Young function $\Phi$,  one can define a
complementary function
\begin{equation}\label{complementaria}
\bar \Phi(s)=\sup_{t>0}\{st-\Phi(t)\}
\end{equation}
for $s\ge 0$. Such $\bar \Phi$ is also a Young function and has the
property that
\begin{equation}\label{preHolder-Orlicz}
st \le C\, \Big[ \Phi(t) +\bar \Phi(s) \Big]
\end{equation}
for all $s,t\ge 0$. Also the $\bar\Phi$-norms  are related to the
$L_{\Phi}$-norms  via the  {\it the generalized   H\"older  
inequality}, namely
\begin{equation}\label{Holder-Orlicz}
\frac1{|E|}\,\int_{E}|f(x)\,g(x)|\,dx \le
2\,\|f\|_{\Phi,E}\,\|g\|_{\bar\Phi,E}.
\end{equation}

In this article we will be particularly interested  in the pair of
Young functions
$$
\Phi_n(t):= t[\log(e+t)]^{n-1} \q \mbox{and}\q \bar\Phi_n(t) \approx
\Psi_n(t):=\exp(t^{\f1{n-1}})-1, \q t\geq 0.
$$
It is the case that the pair $\Phi_n$, $\Psi_n$ satisfies
\eqref{preHolder-Orlicz}, see the article by Bagby \cite{Bagby}, page 887.
Observe that the above function $\Phi_n$ is submultiplicative, a
fact that will be used often in this article. That is,  for $s,t>0$
$$
\Phi_n(st)\le c\,\Phi_n(s)\,\Phi_n(t).
$$
In Section \ref{endpoint} we introduced the   function
$ \Phi^{(m)}:= \overbrace{\Phi\circ \Phi \circ\cdots\circ\Phi}^{m
\mbox{ times}} $
which  is increasing with respect to $m \in \mathbb N$.

\subsection{Some Lemmas}\label{key lemmas} 
We begin by proving some useful general lemmas about averaging
functions.

\begin{lemma}\label{observationYoung}
Let $\Phi$ be any Young function, then for any $f\geq 0$ and any
measurable set $E$
$$1<\|f\|_{\Phi, E} \q \Rightarrow  \q  \|f\|_{\Phi, E}\leq \f1{|E|}\int_E \Phi(f(x))\,dx\, . $$
\end{lemma}

\begin{proof} Indeed, by homogeneity this is equivalent to
$$
\Big\|\frac{f}{\la_{f,E}}\Big\|_{\Phi, E}\leq 1\, ,
$$
where
$$
\la_{f,E}=   \f1{|E|}\int_E \Phi(f(x))\,dx\, ,
$$
which is the same as
$$
\f1{|E|}\int_E \Phi\Big(\frac{f(x)}{\la_{f,E}}\Big)\,dx\leq 1
$$
by definition of the norm \eqref{property0}. In view of
 Property \eqref{property1}, 
it would be enough to show that
$$
\la_{f,E}= \f1{|E|}\int_E \Phi(f(x))\,dx  \geq 1.
$$
But this is the case by definition of the norm (Property
\eqref{property2})
$$\|f\|_{\Phi, E}>1 \q \Longleftrightarrow  \q \f1{|E|}\int_E
\Phi(f(x))\,dx > 1.$$
\end{proof}

The following lemma is key for the main result.
 It should be mentioned that a different proof
of this lemma will appear in a paper by P\'erez, Pradolini, Torres and  Trujillo-Gonz\'alez \cite{PPTT};
see the proof of Theorem 4.1 therein.

\begin{lemma}\label{keylemma}
Let $\Phi$ be a submultiplicative Young function, let\, $m\in\nn$
and let $E$ be any set. Then there is a constant $c$ such that
whenever
\begin{equation}\label{mainassumption}
1 < \prod_{i=1}^m \|f_i\|_{\Phi, E}
\end{equation}
holds, then
\begin{equation}\label{main}
\prod_{i=1}^m \|f_i\|_{\Phi, E}\leq c\,\prod_{i=1}^m \f1{|E|}\int_E \Phi^{(m)} (f_i(x))\,dx \, .
\end{equation}
\end{lemma}

\begin{proof}
a) {\it The  case $m=1$.} This is the content of Lemma
\ref{observationYoung}.

b) {\it The case $m=2$.} Fix functions for which
\eqref{mainassumption} holds:
$$1 < \prod_{i=1}^2 \|f_i\|_{\Phi, E}.
$$
Without loss of generality we may assume that
$$  \|f_1\|_{\Phi, E}\leq \|f_2\|_{\Phi, E}\, .
$$
Observe that by \eqref{mainassumption} we must have
\,$\|f_2\|_{\Phi, E}
>1$.

Suppose first that $1\leq \|f_1\|_{\Phi, E}$,  then \eqref{main}
follows from Lemma \ref{observationYoung}:
$$
1 <  \prod_{i=1}^2 \|f_i\|_{\Phi, E}\leq \prod_{i=1}^2
\f1{|E|}\int_E \Phi (f_i(x))\,dx
$$
with $m=1$ and $c=1$.

Assume now
$$  \|f_1\|_{\Phi, E}\leq  1 \leq \|f_2\|_{\Phi, E}\, .
$$
Then we have by Lemma \ref{observationYoung}, submultiplicativity
and Jensen's inequality
\begin{eqnarray*}
1 &<&  \prod_{i=1}^2 \|f_i\|_{\Phi, E}
\\
&=& \|f_1\|_{\Phi,
E}\,\|f_2\|_{\Phi, E} \\
&=& \big\|f_1\,\|f_2\|_{\Phi, E} \big\|_{\Phi, E}    \\
&\leq & c\,\f1{|E|}\int_E \Phi (f_1(x)\|f_2\|_{\Phi, E})\,dx
\\
&\leq& c\,\f1{|E|}\int_E \Phi (f_1(x))\,dx\,\Phi (\|f_2\|_{\Phi, E})
\\
&\leq& c\,\f1{|E|}\int_E \Phi (f_1(x))\,dx\,\Phi (c\,\f1{|E|}\int_E
\Phi (f_2(x))\,dx\,)
\\
&\leq& c\,\f1{|E|}\int_E \Phi (f_1(x))\,dx\, \f1{|E|}\, \int_E
\Phi^{(2)} (f_2(x))\,dx
\\
 &\leq& c\, \prod_{i=1}^2 \f1{|E|}\int_E \Phi^{(2)} (f_i(x))\,dx\, ,
\end{eqnarray*}
which  is exactly \eqref{main}.

c) {\it The case $m\ge 3$.}  By induction, assuming that the result
holds for the integer $m-1 \geq 2$, we will prove it for $m$. Fix
functions for which \eqref{mainassumption} holds:
$$1 < \prod_{i=1}^m \|f_i\|_{\Phi, E},
$$
and without loss of generality  assume that
$$  \|f_1\|_{\Phi, E}\leq \|f_2\|_{\Phi, E}\leq \dots \leq   \|f_m\|_{\Phi, E}\, .
$$
Observe that we must have \,$\|f_m\|_{\Phi, E} >1$.

If we suppose that $1\leq \|f_1\|_{\Phi, E}$, \, then \eqref{main}
follows directly from Lemma \ref{observationYoung}:
$$
1 <  \prod_{i=1}^m \|f_i\|_{\Phi, E}\leq \prod_{i=1}^m
\f1{|E|}\int_E \Phi (f_i(x))\,dx
$$
with $c=1$ and $\Phi$ instead of $\Phi^{(2)}$.

Assume now that for some integer $k\in \{1,2,\dots, m-1\}$ we have
$$
 \|f_1\|_{\Phi, E}\leq \|f_2\|_{\Phi, E}\leq \cdots \leq \|f_k\|_{\Phi, E}
 \leq 1 \leq \|f_{k+1}\|_{\Phi, E}\leq \cdots \le \|f_m\|_{\Phi, E}\, .
$$
Since
\begin{equation*}
1 <  \prod_{i=1}^m \|f_i\|_{\Phi, E}= \|f_1\|_{\Phi,
E}\,\prod_{i=2}^m \|f_i\|_{\Phi, E},  
\end{equation*}
 we must have $\prod_{i=2}^m \|f_i\|_{\Phi, E}>1$.
Using the induction hypothesis we have
\begin{equation}\label{CP17}
\|f_1\|_{\Phi, E}\,\prod_{i=2}^m \|f_i\|_{\Phi, E} \leq c\,
\|f_1\|_{\Phi, E} \prod_{i=2}^m \f1{|E|}\int_E \Phi^{(m-1)}
(f_i(x))\,dx = \|f_1\,R\|_{\Phi, E}\, ,
\end{equation}
where $R=\prod_{i=2}^m \f1{|E|}\int_E \Phi^{(m-1)} (f_i(x))\,dx$.
Applying Lemma \ref{observationYoung} to the function\, $f_1\,R$ we
obtain by submultiplicativity and Jensen's inequality
\begin{eqnarray*}
\|f_1\,R\|_{\Phi, E} &\leq& c\,\f1{|E|}\int_E \Phi (f_1(x)\,R)\,dx \\
&\leq & c\,\f1{|E|}\int_E
\Phi (f_1(x))\,dx\,\, \Phi( R )\\
& \leq& c\, \f1{|E|}\int_E \Phi (f_1(x))\,dx  \prod_{i=2}^m
\Phi\bigg(\f1{|E|}\int_E \Phi^{(m-1)} (f_i(x))\,dx\!\bigg) \\
&\leq &c\,\f1{|E|}\int_E \Phi (f_1(x))\,dx\, \prod_{i=2}^m
\f1{|E|}\int_E \Phi^{(m)} (f_i(x))\,dx.
\end{eqnarray*}
Combining this result with \eqref{CP17} we deduce
$$
\prod_{i=1}^m \|f_i\|_{\Phi, E} \leq c\,\prod_{i=1}^m \f1{|E|}\int_E
\Phi^{(m)} (f_i(x))\,dx\, ,
$$
thus proving \eqref{main}.
\end{proof}

\subsection{The proof of the endpoint estimates} \label{proofofstrongmaximalthm}


\begin{proof}[\textbf{Proof of Theorem \ref{BSMF}.} ]
By homogeneity, positivity of the operator, and the doubling
property of $\Phi_n$, it is enough to prove
\begin{equation}\label{mainestimate}
 \q\lf|\lf\{x\in\rn:\, \cm_{\crec}(\vec f\,)(x)>1 \r\}\r| \le C
\lf\{\prod_{j=1}^m\int_{\rn} \Phi^{(m)}_n\lf(f_j(x)\r)\,dx
\r\}^{1/m},
\end{equation}
for some constant $C$ independent of the vector of nonnegative
functions $\vec f=(f_1,\cdots, f_m)$.

Let $E = \{x\in\rn:\, \cm_{\crec}(\vec f\,)(x)>1 \}$, then by the
continuity property of the Lebesgue measure we can find a compact
set $K$ such that $K \subset E $ and
$$|K| \leq |E| \leq 2 |K|. $$
Such a compact set $K$   can be covered with a finite collection of
rectangles   $\{R_j\}_{j=1}^N$ such that
\begin{equation}\label{bigger1}
\prod_{i=1}^m   \f 1 {|R_j|} \int _{R_j} f_i(y) \, dy >1,\q
j=1,\cdots,N.
\end{equation}

 We will use the following version of the
C\'ordoba-Fefferman rectangle covering lemma \cite{cf} due to
Bagby (\cite{Bagby} Theorem. 4.1 (C)): there are dimensional
positive constants $\de,c$ and a subfamily $\{\widetilde
R_j\}_{j=1}^{\ell}$ of $\{R_{j}\}_{j=1}^N$ satisfying
$$\bigg|\bigcup_{j=1}^N R_{j}\bigg|\leq c\,  \bigg|\bigcup_{j=1}^\ell \widetilde R_j\bigg|,$$
and
$$
\int_{\bigcup_{j=1}^\ell \widetilde R_j} \exp
\bigg(\de\,\sum_{j=1}^\ell \chi_{\widetilde
R_j}(x)\bigg)^{\f1{n-1}}\,dx \le 2\bigg|\bigcup_{j=1}^\ell
\widetilde R_j\bigg| \, .
$$
Setting $\widetilde E=\bigcup_{j=1}^\ell \widetilde R_j$ and
recalling that $\Psi_n(t)=\exp(t^{\f1{n-1}})-1$ the latter
inequality is
$$
\f1{|\widetilde E|}\int_{\widetilde E} \Psi_n \bigg(
\de\,\sum_{j=1}^\ell \chi_{\widetilde R_j}(x)\bigg) \,dx \leq 1
$$
which is equivalent to
\begin{eqnarray}\label{ee2}
\bigg\|\sum_{j=1}^\ell \chi_{\widetilde R_j}\bigg\|_{\Psi_n,
\widetilde E}\le \frac{1}{\de}
\end{eqnarray}
by the definition of the norm. Now, since
$$|E|\le 2|K| \leq C |\widetilde E|
$$
we can use \eqref{bigger1} and H\"older's inequality as follows
\begin{eqnarray*}
|\widetilde E| &&=  \bigg|\bigcup_{j=1}^\ell \widetilde R_j\bigg| \\
&&\leq \,\sum_{j=1}^\ell |\widetilde R_j|\\
&&\le  \sum_{j=1}^\ell \bigg(\prod_{i=1}^m \int_{\widetilde R_j}
f_i(y)\,dy \bigg)^{\frac1m} \\
&&\leq \bigg(\prod_{i=1}^m \sum_{j=1}^\ell
\int_{\widetilde R_j} f_i(y)\,dy \bigg)^{\frac1m}\\
&&\le \bigg(\prod_{i=1}^m \int_{\bigcup_{j=1}^\ell \widetilde R_j}
\sum_{j=1}^\ell \chi_{\widetilde R_j}(y) f_i(y)\,dy
 \bigg)^{\frac1m}\\
 &&= \bigg(\prod_{i=1}^m \int_{\widetilde E}
\sum_{j=1}^\ell \chi_{\widetilde R_j}(y) f_i(y)\,dy
 \bigg)^{\frac1m}.
\end{eqnarray*}
By this inequality and \eqref{Holder-Orlicz}, we deduce
\begin{eqnarray*}
1\ &\le &\ \prod_{i=1}^m \frac 1{|\widetilde E|} \int_{\widetilde E}
\sum_{j=1}^\ell \chi_{\widetilde R_j}(y) f_i(y)\,dy  \\
& \leq&\
\prod_{i=1}^m \bigg\|\sum_{j=1}^\ell \chi_{\widetilde
R_j}\bigg\|_{\Psi_n, \widetilde E} \|f_i\|_{\Phi_n, \widetilde E}\\
&\leq&\ \prod_{i=1}^m \frac{1}{\de} \|f_i\|_{\Phi_n, \widetilde E} \\
&=&\
\prod_{i=1}^m \Big\|\frac{f_i}{\de}\Big\|_{\Phi_n, \widetilde E}\, .
\end{eqnarray*}
Finally, it is enough to apply Lemma \ref{keylemma} and that
$\Phi_n^{(m)}$ is submultiplicative to conclude the
proof of \eqref{mainestimate}, which is the main part of the
theorem.

We now turn to the claimed
 sharpness of the theorem. In the case  $m=n=2$, we need to show that the
estimate
$$
 \lf|\lf\{x\in \mathbf R^2:\, \cm_{\crec}(f,g)(x)>\al^2\r\}\r|
\le C \lf\{  \int_{\mathbf R^2} \Phi_2
\lf(\f{|f(x)|}{\al}\r)\,dx\r\}^{\!\f12} \lf\{  \int_{\mathbf R^2}
\Phi_2 \lf(\f{|g(x)|}{\al}\r)\,dx\r\}^{\f12}
$$
cannot hold  for $\al>0$ and functions $f,g$ with a constant $C$
independent of these parameters.

For $N=1,2,\dots$, consider the functions
$$f=\chi_{ [0, 1]^2 } \q \textup{and} \q g_N= N \chi_{ [0, 1]^2 } $$
and the parameter $\al=\frac{1}{10}$. Then the left hand side of the
inequality reduces to
\begin{align*}
\Big|\big\{  x\in \rr^2: {\mathcal M }_{\crec} ( f, g_N )(x)  >
\f{1}{100} \big\}\Big| & =   \Big|\big\{x\in \rr^2: M_{\crec} (\chi_{[0,
1]^2 } ) (x) > \frac 1 { 10 \sqrt N } \big\}\Big| \\ & \approx    \sqrt N\,  (
\log N ),
\end{align*}
where the last estimate is a simple calculation concerning the
strong maximal function (i.e., the   case $m=1$) that can be found,
for instance, in \cite[p.\,384, Exercise 10.3.1]{GrMF}.

However, the right hand side is equal to
$$
C(\Phi_2( {1}/{\al}))^{1/2}\,(\Phi_2( {N}/{\al}))^{1/2}= C
(\Phi_2(10))^{1/2} \,(\Phi_2(10N))^{1/2} \approx \sqrt{ N \log N }
$$
and  obviously it cannot control the left hand side for $N$ large.

For general $m$, the vector $\vec f$ with
$$f_1 =f_2=  \cdots= f_{m-1}=\chi_{ [0, 1]^2 } \q {\rm and} \q
f_m= N \chi_{ [0, 1]^2 }$$
also provides a counterexample.
\end{proof}

\section{Interpolation between distributional estimates}\label{interp22}

We have seen that bi(sub)linear operators satisfy certain
distributional estimates that are variations of the usual weak $L^p$
estimates. Multilinear interpolation between a set of restricted
weak type conditions is a well understood topic, but the issue of
multilinear interpolation between more general distributional
estimates has not been studied. We begin this section with the
following interpolation result between distributional estimates.
The  result is not bilinear per se, as the first function remains in
the same space during the interpolation,
 however, it only requires two initial conditions instead of three required
 in the classical real-method bilinear interpolation, see for instance \cite{GrKa}.
 Also the next result can be applied to the maximal function
$\cm_{\crec}$, yielding an $L^1\times L^p$ estimate for it.

\begin{proposition}\label{interpolation1}
Let $T$ be a bisublinear operator. Suppose that there exists $B_1>0$
such that for all $f,\,g\in L_{\Phi_n}^1(\rn)$ and all $\al>0$ we
have
\begin{eqnarray}\label{T:1x1}
&&\lf|\lf\{x\in\rn:\, T(f, g)(x)>\al\r\}\r|
\le\sqrt{B_1\lf\|\Phi_n\lf(\f{|f|}{\sqrt\al}\r)\r\|_{L^1(\rn)}
\lf\|\Phi_n\lf(\f{|g|}{\sqrt\al}\r)\r\|_{L^1(\rn)}} \, .
\end{eqnarray}
Also suppose that there exists $B_2>0$ such that for all $f\in
L_{\Phi_n}^1(\rn)$, $g\in L^\nf(\rn)$ and all $\al>0$,
\begin{eqnarray}\label{T:1xinfty}
\lf|\lf\{x\in\rn:\, T(f, g)(x)>\al\r\}\r| \le B_2
\lf\|\Phi_n\lf(\f{|f|}{\sqrt\al}\r)\r\|_{L^1(\rn)}
\Phi_n\lf(\f{\|g\|_{L^\infty(\rn)}}{\sqrt\al}\r).
\end{eqnarray}
Then, for all $f\in L_{\Phi_n}^1(\rn)$, $g\in L_{\Phi_n}^p(\rn)$
with $p\in(1,\nf)$, and all $\al>0$,
\begin{eqnarray}\label{T:1xp}
&&\lf|\lf\{x\in\rn:\, T(f, g)(x)>\al\r\}\r|\\
&&\qqqq\le
C\lf\{B_1^{\f1{p}}B_2^{\f{p-1}p}\lf\|\Phi_n\lf(\f{|f|}{\sqrt\al}\r)\r\|_{L^1(\rn)}
\lf\|\Phi_n\lf(\f{|g|}{\sqrt\al}\r)\r\|_{L^p(\rn)}\r\}^{\f
p{p+1}},\nonumber
\end{eqnarray}
where $C>0$ depends only on $n$.
\end{proposition}

\begin{proof}
For any $\al>0$, we split $g$ as $g=g_\al+g^\al$ where
$$
g_\al:= g\chi_{\{|g|\le\ez\sqrt{\al/2}\}},\quad \mbox{and }\quad
g^\al:= g\chi_{|g|>\ez\sqrt{\al/2}}\, ,
$$
where $\ez$ is a positive quantity to be determined. Then,
\begin{eqnarray*}
&&\lf|\lf\{x\in\rn:\, T(f, g)(x)>\al\r\}\r|\\
&&\qqq\le\lf|\lf\{x\in\rn:\, T(f, g^\al)(x)>\al/2\r\}\r|
+\lf|\lf\{x\in\rn:\, T(f, g_\al)(x)>\al/2\r\}\r|\\
&&\qqq:=L_1+L_2\, .
\end{eqnarray*}
Since $\Phi_n$ is a strictly increasing function on $(0,\nf)$ and
$p>1$, we have
\begin{eqnarray}\label{e15}
\lf\|\Phi_n\lf(\f{|g^\al|}{\sqrt{\al/2}}\r)\r\|_{L^1(\rn)}
&&=\int_{\{|g|>\ez\sqrt{\al/2}\}} \Phi_n\lf(\f{|g(x)|}{\sqrt{\al/2}}\r)\,dx\\
&&\le\f1{\Phi_n(\ez)^{p-1}}\int_{\{|g|>\ez\sqrt{\al/2}\}} \Phi_n\lf(\f{|g(x)|}{\sqrt{\al/2}}\r)^p\,dx\nonumber\\
&&\le C
\f{\lf\|\Phi_n\lf(\f{|g|}{\sqrt\al}\r)\r\|_{L^p(\rn)}^p}{\Phi_n(\ez)^{p-1}}\,
,\nonumber
\end{eqnarray}
and hence, by \eqref{T:1x1},
\begin{eqnarray*}
L_1&&\le
\sqrt{B_1\lf\|\Phi_n\lf(\f{|f|}{\sqrt{\al/2}}\r)\r\|_{L^1(\rn)}
\lf\|\Phi_n\lf(\f{|g^\al|}{\sqrt{\al/2}}\r)\r\|_{L^1(\rn)}}\\
&&\le C\sqrt{\f{B_1}{\Phi_n(\ez)^{p-1}}
\lf\|\Phi_n\lf(\f{|f|}{\sqrt\al}\r)\r\|_{L^1(\rn)}
\lf\|\Phi_n\lf(\f{|g|}{\sqrt\al}\r)\r\|_{L^p(\rn)}^p}.
\end{eqnarray*}
Also, by \eqref{T:1xinfty},
\begin{eqnarray*}
L_2&&\le \lf\|\Phi_n\lf(\f{|f|}{\sqrt{\al/2}}\r)\r\|_{L^1(\rn)}
\Phi_n\lf(\f{\|g_\al\|_{L^\infty(\rn)}}{\sqrt{\al/2}}\r) \le C B_2
\Phi_n(\ez)\lf\|\Phi_n\lf(\f{|f|}{\sqrt\al}\r)\r\|_{L^1(\rn)}.
\end{eqnarray*}
Again, using $\Phi_n$ is a strictly increasing function on $(0,
\nf)$ with $\Phi_n(0)=0$ and $\Phi_n(\nf)=\nf$, we can choose
$\ez\in(0,\nf)$ such that
$$\Phi_n(\ez)=\lf(\f{B_1 \lf\|\Phi_n\lf(\f{|g|}{\sqrt\al}\r)\r\|_{L^p(\rn)}^p}
{B_2^2
\lf\|\Phi_n\lf(\f{|f|}{\sqrt\al}\r)\r\|_{L^1(\rn)}}\r)^{\f1{p+1}}.$$
For such an $\ez$, both $L_1$ and $L_2$ are bounded by
$$C B_1^{\f1{p+1}}B_2^{1-\f2{p+1}} \lf(\lf\|\Phi_n\lf(\f{|f|}{\sqrt\al}\r)\r\|_{L^1(\rn)}
\lf\|\Phi_n\lf(\f{|g|}{\sqrt\al}\r)\r\|_{L^p(\rn)}\r)^{\f p{p+1}},$$
which proves \eqref{T:1xp}.
\end{proof}

\begin{corollary}\label{JMZ2}
Let $p\in(1,\nf)$. Then, there exists a positive constant $C$
depending only on the dimension $n$ such that for all
 $f\in
L_{\Phi_n^{(2)}}^1(\rn)$, $g\in L_{\Phi_n^{(2)}}^p(\rn)$ and all
$\al>0$,
\begin{eqnarray*}
&& \lf|\lf\{x\in\rn:\, \cm_{\crec}(f, g)(x)>\al\r\}\r|  \\
&&\qq \le  C  \lf\{\int_{\rn}
 \Phi_n^{(2)} \lf(\f{|f(x)|}{\sqrt\al}\r)\,dx\r\}^{\f p{p+1}}
\lf\{\int_{\rn}  \Phi_n^{(2)}
\lf(\f{|g(x)|}{\sqrt\al}\r)^p\,dx\r\}^{\f 1{p+1}}.
\end{eqnarray*}
\end{corollary}

Our last result yields strong type bounds for a bilinear operator
with the initial assumption of two weak type estimates and a third
distributional estimate.  This result provides a generalization of
the well-known real (or Marcinkiewicz) bilinear interpolation
theorem.

\begin{theorem}\label{interpolation2}

Let  $1<s_1<s_2< \infty$  and  $1/s_1 + 1/s_2 = 1/s $. Suppose that
a   bisublinear    operator $T$ maps $L^{ s_1} \times  L^{s_2}  \to
L^{s,\infty}$ with norm $B_1$, it maps $L^{s_2}  \times L^{ s_1} \to
L^{s,\infty}$ with norm $B_2$, it maps  $L^{2s}\times L^{2s}$ to $L^{s,\nf}$ with norm $B$,
 and it also satisfies the following
distributional estimate  at the   endpoint  $(1,1, 1/2)$
$$
\Big| \Big\{ |T(f_1,f_2)|>\la\Big\}\Big| \le {A}\, \bigg( \int
\Phi \Big(\f{f_1}{\sqrt{\la}} \Big)   \, dx \bigg)^{\f12} \bigg(
\int \Phi \Big(\f{f_2}{\sqrt{\la}} \Big)   \, dx \bigg)^{\f12}\, ,
$$
where $\Phi $ is a nonnegative function  that satisfies
 $ \Phi(0)=0$, and  
$$\int_0^1 \lambda^\al \Phi\Big(\frac{1}{\la}\Big)\,d\la<\infty$$
for all $\al>0$. Then, $T:  L^{p_1} \times L^{p_2}  \to  L^p$   for
all indices $p_1,p_2,p$  with $1/p_1+1/p_2=1/p$
and   $(1/p_1, 1/p_2, 1/p)$  is in the open convex hull of the
points $  (1/s_1, 1/s_2, 1/s)$, $(1/s_2, 1/s_1, 1/s)$  and  $(1/s_1,
1/s_1, 2/s_1)$.
\end{theorem}

\begin{proof}
We begin by observing that if $T$ is actually linear in every entry (instead of sublinear), then
the condition that $T$
 maps $L^{2s}\times L^{2s}$ to $L^{s,\nf}$ is redundant as it can be deduced from the
hypotheses that
$T$ maps $L^{ s_1} \times  L^{s_2}  \to
L^{s,\infty}$   and that it maps $L^{s_2}  \times L^{ s_1} \to
L^{s,\infty}$ via bilinear complex interpolation since the point $(1/2s, 1/2s, 1/s)$ lies halfway
between  $(1/s_1,1/s_2,1/s)$, and $(1/s_2,1/s_1,1/s)$.

 We also note that in the desired range of exponents in
the conclusion of the theorem we always have  $s_1/2<p<s<s_2/2$. We
will show the result  holds when $p_1= p_2=2p$  and  $(p, p, p/2)$
is in the claimed open convex hull. Boundedness for the
 remaining triples
follows then by  real bilinear interpolation;  see \cite{GrKa}.  

We fix two functions $f_1,f_2$ in $L^{2p}$ with norm equal to $1$
and write
$$
f_{1,\la} = f_1\chi_{|f_1|>\sqrt{\la}}, \qq f_{1}^{\la} =
f_1\chi_{|f_1|\le \sqrt{\la}}
$$
and likewise for $f_2$.

We now estimate the measure
$$
\big|   \{ |T(f_1,f_2)|> 4\la \}\big| \le
I(\la)+II(\la)+III(\la)+IV(\la) \, ,
$$
where
\begin{eqnarray*}
I(\la) & = & \Big| \Big\{ |T(f_{1,\la},f_{2,\la})|>\la\Big\}\Big|, \\
II (\la)& = & \Big| \Big\{ |T(f_{1,\la},f_{2}^{\la})|>\la\Big\}\Big|, \\
III (\la)& = & \Big| \Big\{ |T(f_{1}^{\la},f_{2,\la})|>\la\Big\}\Big|, \\
IV (\la)& = & \Big| \Big\{
|T(f_{1}^{\la},f_{2}^{\la})|>\la\Big\}\Big|.
\end{eqnarray*}
First we take a look at
$$
\int_0^\nf \la^{p-1} I(\la) \, d\la ,
$$
which is bounded by
\begin{eqnarray*}
&&\int_0^\nf \la^{p-1} A \bigg( \int \Phi
\Big(\f{f_{1,\la}}{\sqrt{\la}} \Big)   \, dx \bigg)^{\f12}
\bigg( \int \Phi \Big(\f{f_{2,\la}}{\sqrt{\la}} \Big)   \, dx \bigg)^{\f12}\, d\la \\
&&\q \le  A\, \prod_{i=1}^2 \bigg( \int_0^\nf \la^{p-1}
\int_{|f_i|>\sqrt{\la}}
\Phi \Big(\f{f_{i}}{\sqrt{\la}} \Big)   \, dx \, d\la \bigg)^{\f12}\\
 &&\q= A\,
 \prod_{i=1}^2 \bigg( \int  \int_{\la=0}^{|f_i(x)|^2}  \la^{p-1}   \Phi \Big(\f{f_{i}}{\sqrt{\la}} \Big)
  d\la \, dx \bigg)^{\f12}\\
   &&\q= 2 A\, \prod_{i=1}^2 \bigg( \int  |f_i(x)|^{2p} \int_{\la=0}^{1}  \la^{2p-1}
\Phi  \Big(\f{1}{{\la}}\Big)
  d\la \, dx \bigg)^{\f12} \\
  &&\q=C_p\, A\, \prod_{i=1}^2 \| f_i \|_{L^{2p}}^p =C_p\, A\, ,
\end{eqnarray*}
where   we have made some simple changes of  variables  and the
 convergence of the integral is due to the fact that   $p>1/2$.

 We now split  $p-s= s(\f{p}{s_1}-\f12) + s(\f{p}{s_2}-\f12)$
 and
 we look at
$$
\int_0^\nf \la^{p-1} II(\la) \, d\la , 
$$
which can be estimated by
\begin{eqnarray*}
&& \int_0^\nf \la^{p-s} B_1^{s } \bigg( \int  | f_{1,\la}  |^{s_1}
\, dx \bigg)^{\f{s}{s_1}}
\bigg( \int |  f_2^\la  |^{s_2}    \, dx \bigg)^{\f{s}{s_2}}\, \f{d\la}{\la} \\
 &&\q \le B_1^{s }   \bigg(\int_0^\nf \la^{s(\f{p}{s_1}-\f12) \f{s_1}{s}}
 \int_{|f_1|>\sqrt{\la}}  | f_1  |^{s_1}  \, dx\,  \f{d\la}{\la} \bigg)^{\f{s}{s_1}}\\
  &&\qq\q\times
  \bigg(\int_0^\nf \la^{s(\f{p}{s_2}-\f12) \f{s_2}{s}}
 \int_{|f_2|\le \sqrt{\la}}  | f_2  |^{s_2}  \, dx\,  \f{d\la}{\la} \bigg)^{\f{s}{s_2}}  \\
  &&\q \le B_1^{s }   \bigg(\int | f_1(x)  |^{s_1} \int_0^{|f_1(x)|^2}  \la^{s_1(\f{p}{s_1}-\f12)  }
  \f{d\la}{\la}   \, dx\,   \bigg)^{\f{s}{s_1}}\\
  &&\qq\q\times
   \bigg(\int | f_2(x)  |^{s_2} \int_{|f_2(x)|^2}^{\nf}  \la^{s_2(\f{p}{s_2}-\f12)  } \f{d\la}{\la}   \, dx\,   \bigg)^{\f{s}{s_2}} \\
   &&\q \le B_1^{s } C(s_1,s_2,p) \|f_1\|_{L^{2p}}^{\f{2ps}{s_1}}  \|f_2\|_{L^{2p}}^{\f{2ps}{s_2}} \\
   &&\q= B_1^{s } C(s_1,s_2,p)
\end{eqnarray*}
and both integrals converge since $s_1<2p<s_2$. The term involving
$III(\la)$ is treated similarly using the bound $B_2$.

 Finally we look at
$$
  \int_0^\nf \la^{p-1} IV(\la) \, d\la,
$$
which is bounded by
\begin{eqnarray*}
&& \int_0^\nf \la^{p-1-s} B^s \bigg( \int  | f_{1}^{\la} |^{2s}  \,
dx \bigg)^{\f{s}{2s}}
\bigg( \int | f_{2}^{\la} |^{2s}    \, dx \bigg)^{\f{s}{2s}}\, d\la \\
&&\q \le B^s \prod_{i=1}^2 \bigg( \int_0^\nf \la^{p-1-s}
\int_{|f_i|\le \sqrt{\la}}
|f_i(x)|^{2s}  \, dx \, d\la \bigg)^{\f12}\\
&&\q = B^s \prod_{i=1}^2 \bigg( \int |f_i(x)|^{2s}
\int_{|f_i(x)|^2}^{\nf}  \la^{p-s-1}
  d\la \, dx \bigg)^{\f12}\\
 &&\q = B^s \prod_{i=1}^2 \bigg( \int |f_i(x)|^{2s} |f_i(x)|^{2p-2s}
   \int_{\la=1}^{\nf}  \la^{p-s-1}    d\la \, dx \bigg)^{\f12} \\
&&\q = B^s\, C(p,s)\, \prod_{i=1}^2 \| f_i \|_{L^{2p}}^p = B^s\,
C(p,s) \, ,
\end{eqnarray*}
where  the
 convergence of the integral is due to the fact that   $p<s$.
\end{proof}

 \begin{remark}\label{nonsymetric}
It is not hard to see that a similar result can be obtained in the
above theorem if instead of the symmetric weak type estimates
$L^{s_1}\times L^{s_2} \to L^{s,\nf}$ and $L^{s_2}\times L^{s_1} \to
L^{s,\nf}$ one has  $L^{t_1}\times L^{t_2} \to L^{s,\nf}$ and
$L^{r_1}\times L^{r_2} \to L^{s,\nf}$ where
$1/t_1+1/t_2=1/r_1+1/r_2=1/s$, $t_1/2<s<t_2/2$ and $r_1/2>s>r_2/2$.
For example, if $r_1\geq t_2$,  one gets strong bounds on the open
triangle with vertices $(1/t_1, 1/t_2, 1/s)$, $(1/r_1, 1/r_2, 1/s)$
and $(1/t_1, 1/t_1, 2/t_1)$.
\end{remark}

\begin{corollary}\label{fulltriangle}
Suppose a bisublinear  operator $T$ maps $L^{ s_1} \times  L^{s_2}  \to
L^{s,\infty}$ for all $1<s_1,s_2, s< \infty$  with  $1/s_1 + 1/s_2 =
1/s $ and
 also satisfies the  endpoint  distributional estimate of Theorem~\ref{interpolation2}.
 Then $T: L^{ p_1} \times  L^{p_2}  \to L^{p}$ for all $1/p_1 + 1/p_2 = 1/p $ with $1<p_1,p_2<\infty$
 and $1/2<p<\nf$.
\end{corollary}
The situation in Corollary~\ref{fulltriangle} arises in the study of the bilinear strong maximal function but also in the the study of
certain commutators of bilinear singular integrals  and pointwise
multiplication with functions in $BMO$.  For a bilinear
Calder\'on-Zygmund operator $T$ as in \cite{GT4} and $b_1,b_2$ in
$BMO$ consider
$$
T_{b_1,b_2}(f_1,f_2)= b_1T(f_1,f_2)-T(b_1f_1,f_2) +
b_2T(f_1,f_2)-T(f_1,b_2f_2).
$$
It was shown in \cite{PT} that $T_{b_1,b_2}: L^{s_1}\times L^{s_2}
\to L^{s}$ for all $1<s_1,s_2, s< \infty$  with  $1/s_1 + 1/s_2 =
1/s $. The proof in \cite{PT} was based on weighted estimates and
cannot be extended to $1/2<s\leq 1$. A question was asked then what
kind of endpoint estimate the operator $T_{b_1,b_2}$ may satisfy.
Later on in \cite{LOPTT} a distributional estimate as in
Theorem~\ref{interpolation2} was obtained  with
$$
\Phi(t) = t(1+\log^+(t)),
$$
but the question still remained open about how to interpolate using
such an estimate. A different method was used in \cite{LOPTT} to
obtain the result $T_{b_1,b_2}: L^{s_1}\times L^{s_2} \to L^{s}$ for
$1/2<s\leq 1$, but we now  see from Corollary~\ref{fulltriangle}
that the result can also be obtained by interpolation. This puts in
evidence that the distributional endpoint estimates achieved in this
article are the appropriate ones from the point of
view of  interpolation. 

 It is also interesting to point out that the distributional estimates we use are not quite $L\log L$-type norms. In fact, even in the linear case and estimate 
$L\log L \to L^{1,\infty}$ together with a strong type $(p,p)$ for $p>1$ do not produce in general $(q,q)$ estimates for 
$1<q<p$, unless the operator in question is a translation invariant one in a compact setting. See \cite{tao}\footnote{We would like to thank M. J. Carro, M. Cwikel, and  N. Kalton  for pointing out this reference and facts and for interesting conversations on the subject.}

\end{document}